\documentclass[9.5pt]{article}
\usepackage{amsfonts}
\usepackage{amssymb}
\usepackage{amsmath,amsfonts,amssymb,amsthm}
\usepackage{geometry}
\usepackage[fontset=windows]{ctex}
\usepackage[T1]{fontenc}

\usepackage{fancyhdr}
\usepackage{geometry}
\usepackage{enumitem}
\usepackage {indentfirst} 
\allowdisplaybreaks[4]
\usepackage[marginal]{footmisc}
\usepackage{setspace}

\def\1{{\bf 1}}

\def \<{\langle}
\def \>{\rangle}

\def \bee{\begin{equation}\label}

\allowdisplaybreaks

\def\bZ{{\mathbb Z}}

\def\bF{{\mathbb F}}

\def\bC{{\mathbb C}}

\def\b{\mathfrak{b}}

\def\1{{\bf 1}}
\def\b1{\mathbbold{1}}

\allowdisplaybreaks

\newtheorem{theorem}{Theorem}[section]
\newtheorem{lemma}[theorem]{Lemma}
\theoremstyle{definition}
\newtheorem{definition}[theorem]{Definition}
\newtheorem{example}[theorem]{Example}

\numberwithin{equation}{section}

\begin{document}
 \renewcommand{\abstractname}{Abstract}  
\begin{center}
{\Large {\bf Transposed Poisson Structures on Two Nambu $3$-Lie Algebras
}}

\vspace{0.3cm} Jingjing Jiang \footnote{Corresponding author.}\\

{\it College of Science, Civil Aviation University of
China, Tianjin 300300, China}\\
{\it Email: jingjingjiang0315@126.com}\\

\vspace{0.3cm}  Chunyi Li \\

{\it College of Science, Civil Aviation University of
China, Tianjin 300300, China}\\
{\it Email: chunyi0227@163.com}\\

\vspace{0.3cm}  Jie Lin \\

{\it Sino-European Institute of Aviation Engineering, Civil Aviation University of
China, Tianjin 300300, China}\\
{\it Email: linj022@126.com}\\
\end{center}

\begin{abstract}
  We describe the $\frac{1}{3}$-derivations and transposed Poisson structures of the Nambu 3-Lie algebras $A_\omega^\delta $ and $ A_{f,k} $. Specifically, we first present that $A_\omega^\delta $ is finitely generated and graded. Then we find that $A_\omega^\delta $  has non-trivial $\frac{1}{3}$-derivations and admits only trivial transposed Poisson structures. The 3-Lie algebra $A_{f,k}$ admits non-trivial transposed Poisson structures.
\end{abstract}

\section{Introduction}
The concept of Poisson algebra, introduced by Sim\'eon Denis Poisson, is that of an associative algebra equipped with a Lie bracket which satisfies the Leibniz rule. This structure originates from the study of Hamiltonian mechanics and Poisson geometry. The Lie bracket, known as the Poisson bracket, exhibits skew-symmetry and satisfies the Jacobi identity. These properties make Poisson algebras particularly useful in the analysis of dynamical systems and geometric structures, such as Poisson manifolds, classical and quantum mechanics, and algebraic geometry. Formally, a vector space $L$ endowed with a bilinear operation $[\cdot,\cdot]$ and a commutative algebra structure with a linear operation $\cdot$ is called a Poisson algebra if it satisfies
$$[x,y\cdot z]=[x,y]\cdot z+y\cdot[x,z],\quad \forall x,y,z\in L.$$
Extensive research has been conducted on Poisson algebras and their related algebraic structures, including Novikov-Poisson algebra \cite{XU1997253}, quiver Poisson algebra \cite{YAO2007570}, F-manifold algebras \cite{fmanifold}, generic Poisson algebras \cite{Kaygorodov03042018,freiheitssatz}, etc.

Recently, Bai, Bai, Guo, and Wu introduced the dual notion of Poisson algebra, termed transposed Poisson algebra, by exchanging the roles of the two binary operations in the Leibniz rule that defines Poisson algebra \cite{BAI2023535}. This concept builds upon the earlier work of Filippov, who in 1998 introduced $\delta$-derivations of Lie algebras as a generalization of traditional derivations \cite{deriva}.  More recently, Ferreira, Kaygorodov, and Lopatkin established a connection between $\frac{1}{n}$-derivations of $n$-Lie algebras and transposed Poisson $n$-Lie algebras \cite{Ferreira2021}. This relationship is utilized to describe transposed Poisson structures on various classes of Lie algebras. For instance, transposed Poisson structures have been described on Galilean type Lie algebras and complex finite-dimensional solvable Lie algebras \cite{KAYGORODOV2023104781}, Witt and Witt-type algebras, Block Lie and Block Lie superalgebras \cite{KAYGORODOV2023167,Kaygorodov2023TransposedPS,kaygorodovwitttype}, Virasoro-type algebras \cite{KAYGORODOV2025105356}, and Lie incidence algebras \cite{KAYGORODOV2024458}. Additionally, Yuan and Hua introduced the notion of $\delta$-biderivations and explored transposed Poisson structures for several classes of Lie algebras \cite{YUANbibi}.

The notion of 3-Lie algebras was first introduced in 1985 by Filippov \cite{nlie}. Since then, this structure has been extensively studied, leading to significant advancements in understanding its properties and applications \cite{STJDF9537B9D96BE10DC17A9B14D3E7FC0D4, SHENG2018256}. For example, module extensions of 3-Lie algebra were studied in \cite{STJDBA793671D3236A0157543B89940F3210}, and various constructions of 3-Lie algebras were provided in \cite{STJDF9537B9D96BE10DC17A9B14D3E7FC0D4}.
The concept of Nambu 3-Lie algebras traces its origins to the work of Nambu, who proposed a generalization of Hamiltonian mechanics involving multiple Hamiltonians \cite{PhysRevD.7.2405}. In recent years, Nambu $3$-Lie algebras have found applications in diverse areas of theoretical physics \cite{PhysRevD.7.2405, takh160}. Nambu introduced the canonical Nambu bracket for triple of classical observables on a three dimensional phase space with coordinates $x, y, z$ as follows
\[
[f_1,f_2,f_3]_\partial=\frac{\partial(f_1,f_2,f_3)}{\partial(x,y,z)}.
\]
We call an abstract $3$-Lie algebra a canonical Nambu $3$-Lie algebra if it is isomorphic to a $3$-Lie algebra whose bracket is defined by the canonical Nambu bracket in the above equation. Bai, Li, and Wang construct two kinds of infinite-dimensional $3$-Lie algebras $A_{\omega}^{\delta}$ and $A_{f,k}$ from a given commutative associative algebra, and show that they are all canonical Nambu 3-Lie algebras in \cite{Bai2017}.

In this paper, we describe the $\frac{1}{3}$-derivations and transposed Poisson structures of the $3$-Lie algebras $A_\omega^\delta $ and $ A_{f,k} $. For $A_{\omega}^{\delta}$, we prove that it is a finitely generated and graded $3$-Lie algebra, and we have compute its $\frac{1}{3}$-derivations.
More precisely, we prove in Theorem \ref{AGRADE} that transposed Poisson structures on $A_\omega^\delta$ are trivial. We show that $A_{f,k}$ admits non-trivial transposed Poisson structure in Theorem \ref{th1}. We finish our article with a concrete example of transports Poisson structure on $A_{f,k}$.

\section{Preliminaries}
In this section, we recall some definitions and known results for studying transposed Poisson structure. All algebras and vector spaces mentioned in this paper are considered over complex field $\bC$.
\begin{definition}\cite{BAI2023535, poissonnlie}\label{poisson3}
Let $L$ be a vector space equipped with two nonzero bilinear operations $\cdot$ and $[\cdot,\cdot,\cdot]$. The triple $(L,\cdot,[\cdot,\cdot,\cdot])$ is called a Poisson $3$-Lie algebra if $(L,\cdot)$ is a commutative associative algebra and $(L,[\cdot,\cdot,\cdot])$ is a $3$-Lie algebra that satisfies the following compatibility condition
\begin{equation}
[x,y,uv]=u[x,y,v]+[x,y,u]v.
\end{equation}
\end{definition}
\begin{definition}\cite{BAI2023535}\label{TP3}
Let $L$ be a vector space equipped with two nonzero bilinear operations $\cdot$ and $[\cdot,\cdot,\cdot]$. The triple $(L,\cdot,[\cdot,\cdot,\cdot])$ is called a transposed Poisson $3$-Lie algebra if $(L,\cdot)$ is a commutative associative algebra and $(L,[\cdot,\cdot,\cdot])$ is a $3$-Lie algebra that satisfies the following compatibility condition
\begin{equation}
3u[x,y,z]=[xu,y,z]+[x,yu,z]+[x,y,zu].
\end{equation}
\end{definition}

\begin{definition}
Let $(L,[\cdot,\cdot,\cdot])$ be a 3-Lie algebra. A transposed Poisson structure on $(L,[\cdot,\cdot,\cdot])$ is a commutative associative multiplication $\cdot$ in $L$ which makes $(L,\cdot,[\cdot,\cdot,\cdot])$ a transposed Poisson 3-Lie algebra.
\end{definition}

\begin{definition}\cite{Ferreira2021}\label{13derivation}
Let $(L,[\cdot,\cdot,\cdot])$ be a 3-Lie algebra, $\phi:L\to L$ be a linear map. Then $\phi$ is a $\frac{1}{3}$-derivation if it satisfies
\begin{equation}\label{deri}
\phi([x,y,z])=\frac{1}{3}([\phi(x),y,z]+[x,\phi(y),z]+[x,y,\phi(z)]).
\end{equation}
\end{definition}

Definitions \ref{TP3} and \ref{13derivation} implies the following key lemma.
\begin{lemma}\label{jhjh}
Let $(L,\cdot,[\cdot,\cdot,\cdot])$ be a transposed Poisson 3-Lie algebra and $z$ an arbitrary element from $L$. Then the left multiplication
$R_z$ in the commutative associative algebra $(L,\cdot)$ gives a $\frac{1}{3}$-derivation on the 3-Lie algebra $(L,[\cdot,\cdot,\cdot])$.
\end{lemma}

The basic example of $\frac{1}{3}$-derivation is the multiplication by an element from the ground field. We call this $\frac{1}{3}$-derivation the trivial $\frac{1}{3}$-derivation.
\begin{theorem}\label{guanxi}
Let $L$ be a $3$-Lie algebra without non-trivial $\frac{1}{3}$-derivation. Then every transposed Poisson structure on $L$ is trivial.
\end{theorem}
The proof of this Theorem \ref{guanxi} is similar to \cite{Ferreira2021}, we omit it.

\begin{lemma}\textup{\cite{FARNSTEINER198833}}
Let $G$ be an abelian group, $L$ be a finitely generated Lie algebra, $V$ be a $G$-graded $L$-module. Then
\[
\mathrm{Der}_F(L,V)=\bigoplus_{g\in G}{\mathrm{Der}_F(L,V)}_g.
\]
\end{lemma}
Next, we will extend this lemma to the case of $\frac{1}{3}$-derivations on finitely generated $G$-graded $3$-Lie algebras.

Let $G$ be an abelian group, $L=\bigoplus\limits_{g\in G}{L_g}$ a $G$-graded $3$-Lie algebra, we say that a $\frac{1}{3}$-derivation $\varphi$ has degree $g$ (and denoted by $deg(\varphi)$) if $\varphi(L_h)=L_{g+h}$. Denote by $\mathrm{Der}_{\frac{1}{3}}(L)$ the vector space of all $\frac{1}{3}$-derivations and write $\mathrm{Der}_{\frac{1}{3}}(L)_g=\{\varphi \in \mathrm{Der}_{\frac{1}{3}}(L) | deg(\varphi)=g\}$. We give the following theorem, it is crucial in our work.

\begin{theorem}\label{th:fc}
Let $G$ be an abelian group, $L=\bigoplus\limits_{g\in G}{L_g}$ be a finitely generated $G$-graded $3$-Lie algebra. Then
\[\mathrm{Der}_{\frac{1}{3}}(L)=\bigoplus_{g\in G}{\mathrm{Der}_{\frac{1}{3}}(L)_g}.\]
\end{theorem}
\begin{proof}
For any $g \in G$, we let $\pi_g:L \to L_g$ denote the canonical projection. Since $L$ is finitely generated, there is a finite subset $S \subset L$ generating $L$. Let $\varphi:L\to V$ be a $\frac{1}{3}$-derivation. Then there are finite sets $Q,R\subset G$ such that
\begin{equation}
S\subset \sum_{g \in Q}{L_g} \quad and \quad  \varphi(S) \subset \sum_{g \in R}{L_g}.\label{eq:jh}
\end{equation}

For $g \in G$, put $\varphi_g:=\sum\limits_{h \in G}{\pi_{g+h} \mathbin{\circ} \varphi \mathbin{\circ} \pi_h}$. Since for $x_r \in L_r$, $x_s \in L_s$ and $x_t \in L_t$, we have
\begin{align*}
\varphi_g([x_r,x_s,x_t])=&\sum_{h \in G}{\pi_{g+h} \mathbin{\circ} \varphi \mathbin{\circ} \pi_h}([x_r,x_s,x_t]) \\
=&\pi_{g+r+s+t} \mathbin{\circ} \varphi ([x_r,x_s,x_t])\\
=&\frac{1}{3} \pi_{g+r+s+t}([\varphi(x_r),x_s,x_t])+\frac{1}{3} \pi_{g+r+s+t}([x_r,\varphi(x_s),x_t])+\frac{1}{3} \pi_{g+r+s+t}([x_r,x_s,\varphi(x_t)])\\
=&\frac{1}{3}[\pi_{g+r}(\varphi(x_r)),x_s,x_t]+\frac{1}{3}[x_r,\pi_{g+s}(\varphi(x_s)),x_t]+\frac{1}{3}[x_r,x_s,\pi_{g+t}(\varphi(x_t))]\\
=&\frac{1}{3}[\varphi_g(x_r),x_s,x_t]+\frac{1}{3}[x_r,\varphi_g(x_s),x_t]+\frac{1}{3}[x_r,x_s,\varphi_g(x_t)],
\end{align*}
it follows that $\varphi_g$ is contained in $\mathrm{Der}_{\frac{1}{3}}(L)_g$.\\

Let $T:=\{g-h;h \in Q,g \in R\}$. Then $T$ is finite and we obtain, observing Eq.\eqref{eq:jh}, for $y \in S$
\begin{align*}
\varphi(y)=& \sum_{g\in R}{\pi_g\mathbin{\circ}\varphi(y)}\\
=&\sum_{g\in R}{\sum_{h\in Q}{\pi_g\mathbin{\circ}\varphi\mathbin{\circ}\pi_h(y)}}\\
=&\sum_{h\in Q}{(\sum_{g\in R}{\pi_{(g-h)+h}\mathbin{\circ}\varphi\mathbin{\circ}\pi_h(y)})}\\
=&\sum_{h\in Q}{(\sum_{q\in T}{\pi_{q+h}\mathbin{\circ}\varphi\mathbin{\circ}\pi_h(y)})}\\
=&\sum_{q\in T}{\sum_{h\in Q}{\pi_{q+h}\mathbin{\circ}\varphi\mathbin{\circ}\pi_h(y)}}\\
=&\sum_{q\in T}{\sum_{h\in G}{\pi_{q+h}\mathbin{\circ}\varphi\mathbin{\circ}\pi_h(y)}}\\
=&\sum_{q\in T}{\varphi_q(y)}.
\end{align*}
This shows that the $\frac{1}{3}$-derivation $\varphi$ and $\sum\limits_{q\in T}{\varphi_q(y)}$ coincide on $S$. As $S$ generates $L$, we obtain $\varphi=\sum\limits_{q\in T}{\varphi_q}$.
\end{proof}

\section{Transposed Poisson structures on 3-Lie algebra $A_{\omega}^{\delta}$}

As given in \cite{Bai2017}, let $A$ be a commutative associative algebra with a basis $\{L_r,M_r\mid r \in \bZ\}$ in the multiplication
\begin{equation}
L_rL_s=L_{r+s},\quad M_rM_s=M_{r+s},\quad L_rM_s=0, \quad \forall r,s\in \bZ, \label{eq:cf}
\end{equation}
$\delta$ be a derivation of $A$ defined by
\[\delta(L_r)=rL_r,\quad\delta(M_r)=rM_r,\]
and let $\omega$ be an involution of $A$ defined by
\[\omega(L_r)=M_{-r},\quad \omega(M_r)=L_{-r}.\]
In \cite{Bai2017}, the authors utilized Theorem 3.3 from \cite{STJDF9537B9D96BE10DC17A9B14D3E7FC0D4} to obtain the following 3-Lie algebra
\begin{equation*}
\begin{split}
[L_r,L_s,M_t]_{\omega}&=(s-r)L_{r+s-t},\\
[L_r,M_s,M_t]_{\omega}&=(t-s)M_{s+t-r},\\
[L_r,L_s,L_t]_{\omega}&=[M_r,M_s,M_t]_{\omega}=0
\end{split}
\end{equation*}
for any $r,s,t \in \bZ$, and proved that this 3-Lie algebra is a simple canonical Nambu 3-Lie algebra.

Let $M_r=M_{-r}$ for any $r \in \bZ$, we have the following definition of $3$-Lie algebra $A_{\omega}^{\delta}$.
\begin{definition}
$A_{\omega}^{\delta}$ is a $3$-Lie algebra with a basis $\{L_r,M_r\mid r \in \bZ\}$, subject to the relation
\begin{equation}\label{2eq1}
\begin{split}
[L_r,L_s,M_t]&=(s-r)L_{r+s+t},\\
[L_r,M_s,M_t]&=(s-t)M_{r+s+t},\\
[L_r,L_s,L_t]&=[M_r,M_s,M_t]=0
\end{split}
\end{equation}
for $r,s,t\in \bZ$.
\end{definition}
Now we will prove that $A_{\omega}^{\delta}$ is finitely generated and graded, i.e., $A_{\omega}^{\delta}=\bigoplus\limits_{k\in \bZ}{(A_{\omega}^{\delta})_k}$ , where $(A_{\omega}^{\delta})_k=\bC L_k\bigoplus \bC M_k$.
\begin{lemma}
The 3-Lie algebra $A_{\omega}^{\delta}$ is finitely generated and $\bZ$-graded.
\end{lemma}

\begin{proof}
We will prove that $A_{\omega}^{\delta}$ is generated by $L_{-1},L_0,L_1,M_{-1},M_0$ and $M_1$. In fact, for any $n\in {\bZ}^+$, we have
\begin{align*}
&L_n=(1-0)L_{0+1+(n-1)}=[L_0,L_1,M_{n-1}],\\
&M_n=(1-0)M_{(n-1)+1+0}=[L_{n-1},M_1,M_0],\\
&L_{-n}=(0-(-1))L_{(-1)+0+(-n+1)}=[L_{-1},L_0,M_{-n+1}],\\
&M_{-n}=(0-(-1))M_{(-n+1)+0+(-1)}=[L_{-n+1},M_0,M_{-1}],
\end{align*}
then by induction on $n$, we have $A_{\omega}^{\delta}$ is finitely generated by $L_{-1},L_0,L_1,M_{-1},M_0$ and $M_1$.

By Eq.\eqref{2eq1}, it is obvious that $A_{\omega}^{\delta}$ is $\bZ$-graded, i.e., $A_{\omega}^{\delta}=(A_{\omega}^{\delta})_k$, where $(A_{\omega}^{\delta})_k=\bC L_k \bigoplus \bC M_k$ for any $k\in \bZ$.
\end{proof}
\begin{theorem}\label{3.5}
The space of all $\frac{1}{3}$-derivations of $A_{\omega}^{\delta}$ is
$$\mathrm{Der}_\frac{1}{3}(A_{\omega}^{\delta})=\bigoplus_{k\in \bZ}{\bC D_k},$$
where $D_k$ is a $\frac{1}{3}$-derivation with $\deg(D_k)=k$ such that for any $r\in \bZ$,
\begin{equation}\label{2eq2}
\begin{cases}
&D_k(L_r)=L_{r+k},\\
&D_k(M_r)=M_{r+k}.
\end{cases}
\end{equation}
\end{theorem}

\begin{proof}
By Theorem \ref{th:fc}, we have $\mathrm{Der}_{\frac{1}{3}}({A_{\omega}^{\delta}})=\bigoplus\limits_{k\in\bZ}{(\mathrm{Der}_{\frac{1}{3}}(A_{\omega}^{\delta}))_k}$. To determine $\mathrm{Der}_{\frac{1}{3}}({A_{\omega}^{\delta}})$, we need to calculate every homogeneous subspace $(\mathrm{Der}_{\frac{1}{3}}(A_{\omega}^{\delta}))_k$.
Let $\varphi_k \in (\mathrm{Der}_{\frac{1}{3}}(A_{\omega}^{\delta}))_k$, where $k$ is a fixed integer. We can assume
\[\varphi_k(L_r)=a_{k,r}L_{r+k}+b_{k,r}M_{r+k},\]
\[\varphi_k(M_r)=c_{k,r}L_{r+k}+d_{k,r}M_{r+k}.\]
For simplicity, we omit the subscript $k$ in coefficients.\\

Applying $\varphi_k$ to equation $[L_r,L_s,M_t]=(s-r)L_{r+s+t}$, we obtain
\begin{align*}
&\varphi_k([L_r,L_s,M_t])=(s-r)\varphi_k(L_{r+s+t})=(s-r)a_{r+s+t}L_{r+s+t+k}+(s-r)b_{r+s+t}M_{r+s+t+k}\\
=&\frac{1}{3}([\varphi_k(L_r),L_s,M_t]+[L_r,\varphi_k(L_s),M_t]+[L_r,L_s,\varphi_k(M_t)])\\
    =&\frac{1}{3}(a_r[L_{r+k},L_s,M_t]+b_r[M_{r+k},L_s,M_t]+a_s[L_r,L_{s+k},M_t]+b_s[L_r,M_{s+k},M_t]+c_t[L_r,L_s,L_{t+k}]+d_t[L_r,L_s,M_{t+k}])\\
    =&\frac{1}{3}((s-r-k)a_r+(s+k-r)a_s+(s-r)d_tL_{r+s+t+k})+\frac{1}{3}(-(r+k-t)b_r+(s+k-t)b_s)M_{r+s+t+k},
\end{align*}
comparing coefficients of the basis elements we obtain
\begin{align}
&3(s-r)a_{r+s+t} = (s-r-k)a_r+(s+k-r)a_s+(s-r)d_t, \label{2eq3} \\
&3(s-r)b_{r+s+t} = (s+k-t)b_s-(r+k-t)b_r. \label{2eq4}
\end{align}

Applying $\varphi_k$ to equation $[L_r,M_s,M_t]=(s-t)M_{r+s+t}$, we obtain
\begin{align*}
&\varphi_k([L_r,M_s,M_t])=(s-t)\varphi_k(M_{r+s+t})=(s-t)c_{r+s+t}L_{r+s+t+k}+(s-t)d_{r+s+t}M_{r+s+t+k}\\
=&\frac{1}{3}([\varphi_k(L_r),M_s,M_t]+[L_r,\varphi_k(M_s),M_t]+[L_r,M_s,\varphi_k(M_t)])\\
    =&\frac{1}{3}(a_r[L_{r+k},M_s,M_t]+b_r[M_{r+k},M_s,M_t]+c_s[L_r,L_{s+k},M_t]+d_s[L_r,M_{s+k},M_t]+c_t[L_r,M_s,L_{t+k}]+d_t[L_r,M_s,M_{t+k}])\\
    =&\frac{1}{3}((s+k-r)c_s-(t+k-r)c_t)L_{r+s+t+k}+\frac{1}{3}((s-t)a_r+(s+k-t)d_s+(s-t-k)d_t)M_{r+s+t+k},
\end{align*}
comparing coefficients of the basis elements we obtain that
\begin{align}
&3(s-t)c_{r+s+t} = (s+k-r)c_s-(t+k-r)c_t, \label{2eq5} \\
&3(s-t)d_{r+s+t} = (s-t)a_r+(s+k-t)d_s+(s-t-k)d_t. \label{2eq6}
\end{align}
In Eq.\eqref{2eq3}, let $s=1$, $r=-1$, we have
\begin{align}
&6a_t-2d_t = (2-k)a_{-1}+(2+k)a_1.\label{2eq9}
\end{align}
In Eq.\eqref{2eq6}, let $s=1$, $t=-1$, we have
\begin{align}
&6d_r-2a_r = (2+k)d_1+(2-k)d_{-1}.\label{2eq10}
\end{align}
By Eq.\eqref{2eq9} and \eqref{2eq10}, we have
\begin{equation}\label{adad}
a_t=\frac{(2-k)a_{-1}+(2+k)a_1+3((2+k)d_1+(2-k)d_{-1})}{16}
\end{equation}
and
$$ d_t=\frac{3((2-k)a_{-1}+(2+k)a_1)+(2+k)d_1+(2-k)d_{-1}}{16},$$
thus both $a_t$ and $d_t$ are numbers that are independent of $t$, we denote them by $a$ and $d$ respectively. By Eq.\eqref{adad}, we have $a=\frac{4a+12d}{16}$, that is $a=d$.

Applying $\varphi_k$ to equation $[L_r,L_s,L_t]=0$, we obtain
\begin{align*}
&\varphi_k([L_r,L_s,L_t])=0\\
=&\frac{1}{3}([\varphi_k(L_r),L_s,L_t]+[L_r,\varphi_k(L_s),L_t]+[L_r,L_s,\varphi_k(L_t)])\\
    =&\frac{1}{3}(b_r[M_{r+k},L_s,L_t]+b_s[L_r,M_{s+k},L_t]+b_t[L_r,L_s,M_{t+k}])\\
    =&\frac{1}{3}(-(s-t)b_r-(t-r)b_s+(s-r)b_t)L_{r+s+t+k}.
\end{align*}
It follows that
\begin{equation}\label{2eq7}
-(s-t)b_r-(t-r)b_s+(s-r)b_t = 0,
\end{equation}
let $r=t+2$ and $s=t+1$ in Eq.\eqref{2eq7}, we have $$b_{t+2}-b_{t+1}=b_{t+1}-b_t.$$ Obviously, $b_t$ is an arithmetic sequence,
implying that $b_t=b_1+(t-1)(b_2-b_1)$. By substituting it into Eq.\eqref{2eq4}, we obtain
\[3(s-r)b_1+3(s-r)(r+s+t-1)(b_2-b_1)=(s-r)b_1+(s-r)(s+r+k-t-1)(b_2-b_1),\]
let $r=1$, $s=2$, we have
$2b_1+4(t+1)(b_2-b_1)=k(b_2-b_1)$ holds for any $t\in\bZ$. Therefore, $b_1=b_2=0$, and consequently, $b_r=0$ for any $r\in\bZ$.

Applying $\varphi_k$ to equation $[M_r,M_s,M_t]=0$, we obtain
\begin{align*}
&\varphi_k([M_r,M_s,M_t])=0\\
=&\frac{1}{3}([\varphi_k(M_r),M_s,M_t]+[M_r,\varphi_k(M_s),M_t]+[M_r,M_s,\varphi_k(M_t)])\\
    =&\frac{1}{3}(c_r[L_{r+k},M_s,M_t]-c_s[L_{s+k},M_r,M_t]-c_t[L_{t+k},M_s,M_r])\\
    =&\frac{1}{3}((s-t)c_r-(r-t)c_s-(s-r)c_t)M_{r+s+t+k}.
\end{align*}
Therefore, we have
\begin{align}
&(s-t)c_r-(r-t)c_s-(s-r)c_t = 0. \label{2eq8}
\end{align}

Observing Eq.\eqref{2eq7} and \eqref{2eq8}, it can be found that the conditions satisfied by $c_t$ and $b_t$ are the same. Similarly, $c_r=0$ for all $r\in\bZ$.
Thus, $\varphi_k$ has the form
$$\begin{cases}
&\varphi_k(L_r)=\alpha L_{r+k},\\
&\varphi_k(M_r)=\alpha M_{r+k},
\end{cases}$$
where $\alpha \in \bC$.

Let $D_k$ be a $\frac{1}{3}$-derivation defined by
$$\begin{cases}
&D_k(L_r)=L_{r+k},\\
&D_k(M_r)=M_{r+k},
\end{cases}$$
then $\mathrm{Der}_\frac{1}{3}(A_{\omega}^{\delta})=\bigoplus\limits_{k\in \bZ}{\bC D_k}$.

\end{proof}
\begin{theorem}\label{AGRADE}
The $3$-Lie algebra $A_{\omega}^{\delta}$ does not admits non-trivial transposed Poisson structures.
\end{theorem}

\begin{proof}
Let $(A_{\omega}^{\delta},\cdot,[\cdot,\cdot,\cdot])$ be a transposed Poisson structure defined on the $3$-Lie algebra $A_{\omega}^{\delta}$. Then by Lemma \ref{jhjh}, for any element of $x \in A_{\omega}^{\delta}$, we have the operator of multiplication $\varphi_x(y)=x\cdot y$ is a $\frac{1}{3}$-derivation. Next, we omit the multiplication symbol $\cdot$ for simplicity. By Lemma \ref{3.5}, we set
$$\varphi_{L_i}=\sum_{k\in\bZ}{\alpha_{i,k}D_k},$$
$$\varphi_{M_i}=\sum_{k\in\bZ}{\beta_{i,k}D_k}.$$
We have
\begin{equation*}
L_iM_j=\varphi_{L_i}(M_j)=\sum_{k\in\bZ}{\alpha_{i,k}M_{k+j}},
\end{equation*}
on the other hand,
\begin{equation*}
M_jL_i=\varphi_{M_j}(L_i)=\sum_{k\in\bZ}{\beta_{i,k}L_{k+i}}.
\end{equation*}
Due to the commutative property, we have $L_iM_j=M_jL_i$, so $\alpha_{i,k}=0$, $\beta_{i,k}=0$ for any $i,k\in \bZ.$
Therefore, $A_{\omega}^{\delta}$ admits only trivial transpose Poisson structures.
\end{proof}

\section{Transposed Poisson structures on 3-Lie algebra $A_{f,k}$}
Similarly, using the commutative associative algebra $A$ given in Eq.(\ref{eq:cf}), for each $k \in \bZ$, let $d_k$ be a derivation of $A$ defined by
\[d_k(aL_r+bM_t)=arL_{k+r}, \quad \forall a,b \in \bC,\quad r,t\in \bZ.\]
In \cite{Bai2017}, the authors proved that through $d_k$, each integer $k$ induces a Lie algebra structure $[\cdot,\cdot]_k$ on $A$ defined by
\[[u,v]_k=d_k(u)v-vd_k(v), \quad \forall u,v \in A.\]
By Theorem 3.3 in \cite{Bai2010RealizationsO3}, the authors obtain the following $3$-Lie algebra.
\begin{definition}\cite{Bai2017}
Let $f:A\to\bF$ be a linear function with $f(L_r)=0$ for all $r\in\bZ$, and there exists some $t$ such that $f(M_t)\neq 0$.
Then for any integer $k$, $(A_{f,k},[\cdot,\cdot,\cdot])$ is a canonical Nambu 3-Lie algebra, where $[\cdot,\cdot,\cdot]$ is defined by
\begin{equation}\label{eq:1}
[L_r,L_s,M_t]=f(M_t)(r-s)L_{r+s+k},\quad \forall\:L_r,L_s,M_t\in A.
\end{equation}
\end{definition}

\begin{theorem}
Let $D$ be a $\frac{1}{3}$-derivation of $A_{f,k}$, then there exists a sequence of complex numbers $h,c_i,d_{r,j} \in \bC$ such that
\begin{equation}\label{eq:2}
D(L_r)=hL_r,
\end{equation}
\begin{equation}
D(M_r)=f(M_r)\sum_{i\in \bZ}{c_iL_i}+\sum_{j\in \bZ}{d_{r,j}M_j},
\end{equation}
where $h$ and $d_{r,j}$ satisfy
\begin{equation}\label{tj}
\sum_{j\in \bZ}f(M_j)d_{r,j}=hf(M_r).
\end{equation}
\end{theorem}

\begin{proof}
Since $(A_{f,k},[\cdot,\cdot,\cdot])$ is not finitely generated, we can not use Theorem \ref{th:fc}.

Let $D$ be a $\frac{1}{3}$-derivation of $A_{f,k}$, assume that
$$
D(L_r)=\sum\limits_{i\in \bZ}{a_{r,i}L_i}+\sum\limits_{j\in \bZ}{b_{r,j}M_j},$$
$$D(M_r)=\sum\limits_{i\in \bZ}{c_{r,i}L_i}+\sum\limits_{j\in \bZ}{d_{r,j}M_j}.
$$

Applying $D$ to equation $[L_r,L_s,M_t]=f(M_t)(r-s)L_{r+s+k}$, we obtain
\begin{equation}\label{rski}
\begin{split}
&D([L_r,L_s,M_t])=\sum\limits_{i\in \bZ}{f(M_t)(r-s)a_{r+s+k,i}L_i}+\sum\limits_{j \in \bZ}{f(M_t)(r-s)b_{r+s+k,j}M_j}\\
=&\frac{1}{3}([D(L_r),L_s,M_t]+[L_r,D(L_s),M_t]+[L_r,L_s,D(M_t)])\\
=&\frac{1}{3}\sum\limits_{i\in \bZ}{f(M_t)(i-s)a_{r,i}L_{i+s+k}}
+\frac{1}{3}\sum\limits_{i\in \bZ}{f(M_t)(r-i)a_{s,i}L_{r+i+k}}
+\frac{1}{3}\sum\limits_{j\in \bZ}{f(M_j)(r-s)d_{t,j}L_{r+s+k}},
\end{split}
\end{equation}
comparing coefficients of $M_j$, we have
\[f(M_t)(r-s)b_{r+s+k,j}=0.\]
Since there exists a $t$ such that $f(M_t)\neq 0$, we have $b_{r+s+k,j}=0$ when $r\neq s$,
then
\begin{equation}\label{eq:3}
b_{i,j} = 0 \quad\text{for any} ~~ i,j\in \bZ.
\end{equation}
Comparing coefficients of $L_{r+s+k}$, we have
\begin{equation}\label{eq:4}
3f(M_t)(r-s)a_{r+s+k,r+s+k} = f(M_t)(r-s)a_{r,r} + f(M_t)(r-s)a_{s,s} + \sum_{j \in \bZ} f(M_j)(r-s)d_{t,j}.
\end{equation}

There exists a $t$ such that $f(M_{t})\neq0$, it follows from Eq.\eqref{eq:4} that
\[3(r-s)a_{r+s+k,r+s+k}=(r-s)a_{r,r}+(r-s)a_{s,s}+\frac{\sum_{j \in \bZ} f(M_j)(r-s)d_{t,j}}{f(M_{t})}.\]
When $r\neq s$,
\begin{equation}\label{ahahaha}
3a_{r+s+k,r+s+k}=a_{r,r}+a_{s,s}+\frac{\sum_{j \in \bZ} f(M_j)d_{t,j}}{f(M_{t})},
\end{equation}
which implies $\frac{\sum_{j \in \bZ} f(M_j)d_{t,j}}{f(M_{t})}$ is a number that is independent of $t$, we denote it by $h$, i.e.,
$$h=\frac{\sum_{j\in\bZ}{f(M_j)d_{t,j}}}{f(M_{t})}$$
for any $t$ such that $f(M_t)\neq 0$. By Eq.\eqref{eq:4}, when $f(M_t)=0$,
$$\sum_{j \in \bZ} f(M_j)d_{t,j}=0=hf(M_t).$$
So for any $r\in \bZ$, Eq.\eqref{tj} holds.

By Eq.\eqref{ahahaha}, we have
\begin{equation}\label{abouta}
3a_{r+s+k,r+s+k}=a_{r,r}+a_{s,s}+h,
\end{equation}
let $s=-k$, we obtain $$a_{r,r}=\frac{a_{-k,-k}+h}{2},\quad\forall r\neq -k,$$
substituting it into Eq.\eqref{abouta}, we have $a_{-k,-k}=h,$ then
\begin{equation}\label{eq 4.10}
a_{r,r}=h \quad \text{for any}\quad r\in \bZ.
\end{equation}
By Eq.\eqref{rski}, when $i\neq 0$, comparing the coefficients of $L_{r+s+k+i}$, we have
\begin{equation}\label{eq:5}
3f(M_t)(r-s)a_{r+s+k,r+s+k+i}=f(M_t)(r+i-s)a_{r,r+i}+f(M_t)(r-i-s)a_{s,s+i},
\end{equation}
let $s=-k$, we obtain
\[3(r+k)a_{r,r+i}=(r+k+i)a_{r,r+i}+(r+k-i)a_{-k,-k+i}.\]
Hence for any $i \neq 0$, when $r\neq \frac{i-2k}{2}$
\begin{equation}\label{aiaiai}
a_{r,r+i}=\frac{r+k-i}{2r+2k-i}a_{-k,-k+i}.
\end{equation}
 Then when $s\neq \frac{i-2k}{2}$ and $r+s\neq \frac{i-4k}{2},$
\[a_{s,s+i}=\frac{s+k-i}{2s+2k-i}a_{-k,-k+i}, \]
\[a_{r+s+k,r+s+k+i}=\frac{r+s+2k-i}{2r+2s+4k-i}a_{-k,-k+i}.\]
Substituting the above three equations into Eq.\eqref{eq:5}, we obtain
$$\left( \frac{3(r-s)(r+s+2k-i)}{2r+2s+4k-i}-\frac{(r+i-s)(r+k-i)}{2r+2k-i}- \frac{(r-s-i)(s+k-i)}{2s+2k-i}\right)a_{-k,-k+i}=0$$
when $r\neq \frac{i-2k}{2}$, $s\neq \frac{i-2k}{2}$ and $r+s\neq \frac{i-4k}{2}$.

Let $r=2i-k,s=-i-k$, since $i\neq 0$, then $r\neq \frac{i-2k}{2}$,$s\neq \frac{i-2k}{2}$ and $r+s\neq \frac{i-4k}{2}$, then we have
$$-\frac{8}{3}ia_{-k,-k+i}=0,$$
so $a_{-k,-k+i}=0$ for $i\neq 0$. By Eq.\eqref{aiaiai}, $a_{r,r+i}=0$ when $r\neq \frac{i-2k}{2}$. There exists $r,s$ such that $r\neq \frac{i-2k}{2},s\neq \frac{i-2k}{2},r\neq s$ and $r+s+k=\frac{i-2k}{2}$, then by  Eq.\eqref{eq:5}, we obtain
\begin{equation}\label{eq:4.13}
a_{r,r+i}=0 \quad \text{for any} \quad r\in \bZ.
\end{equation}

By Eq.(\ref{eq:3}), (\ref{eq 4.10}) and (\ref{eq:4.13}), we have for any $r \in \bZ$
\[D(L_r)=hL_r.\]

Applying $D$ to equation $[L_r,M_s,M_t]=0$, we obtain
\begin{align*}
&D([L_r,M_s,M_t])=0\\
=&\frac{1}{3}([D(L_r),M_s,M_t]+[L_r,D(M_s),M_t]+[L_r,M_s,D(M_t)])\\
=&\frac{1}{3}\sum\limits_{i\in \bZ}{f(M_t)(r-i)c_{s,i}L_{r+i+k}}-\frac{1}{3}\sum\limits_{i\in \bZ}{f(M_s)(r-i)c_{t,i}L_{r+i+k}},
\end{align*}
comparing coefficients of the basis elements we obtain that
\[f(M_t)(r-i)c_{s,i}-f(M_s)(r-i)c_{t,i}=0,\]
then
\[f(M_t)c_{s,i}=f(M_s)c_{t,i}\]
for any $s,t,i\in \bZ$, so there exists a series of number $c_i$ such that
$$c_{s,i}=c_if(M_s)$$
for any $s\in \bZ$,
then
\begin{equation*}
D(M_r)=\sum_{i\in\bZ}{c_{r,i}L_i}+\sum_{j\in\bZ}{d_{r,j}M_j}=f(M_r)\sum_{i\in\bZ}{c_iL_i}+\sum_{j\in\bZ}{d_{r,j}M_j}.
\end{equation*}

For any $L_r,L_s,L_t \in A_{f,k}$,
\begin{align*}
0=D([L_r,L_s,L_t])=&\frac{1}{3}([D(L_r),L_s,L_t]+[L_r,D(L_s),L_t]+[L_r,L_s,D(L_t)])\\
=&\frac{1}{3}\sum_{j\in \bZ}{b_{r,j}[M_j,L_s,L_t]}+\frac{1}{3}\sum_{j\in \bZ}{b_{s,j}[L_r,M_j,L_t]}+\frac{1}{3}\sum_{j\in \bZ}b_{t,j}{[L_r,L_s,M_j]}=0.
\end{align*}

For any $M_r,M_s,M_t \in A_{f,k}$,
\begin{align*}
0=D([M_r,M_s,M_t])=&\frac{1}{3}([D(M_r),M_s,M_t]+[M_r,D(M_s),M_t]+[M_r,M_s,D(M_t)])=0
\end{align*}
holds naturally.

In summary, we have
$$
D(L_r)=hL_r,$$
$$D(M_r)=f(M_r)\sum\limits_{i\in \bZ}{c_iL_i}+\sum\limits_{j\in \bZ}{d_{r,j}M_j},
$$
where $d_{r,j}$ and $h$ satisfy
$$\sum_{j\in \bZ}f(M_j)d_{r,j}=hf(M_r)$$
for any $r\in\bZ$.
\end{proof}

\begin{theorem} \label{th1}
Let $(A_{f,k}, \cdot, [\cdot,\cdot,\cdot])$ be a transposed Poisson structure defined on the 3-Lie algebra $A_{f,k}$, then there exists a sequence of complex numbers $\alpha$, $c_{p}$ and $d_{i,j,q}$ such that the multiplication on $(A_{f,k}, \cdot)$ is given by
\begin{align*}
L_i \cdot L_j &= 0, \\
L_i \cdot M_j &= \alpha f(M_j) L_i, \\
M_i \cdot M_j &= f(M_i)f(M_j) \sum_{p \in \mathbb{Z}} c_{p} L_p + \sum_{q \in \mathbb{Z}} d_{i,j,q} M_q,
\end{align*}
where $\alpha$, $c_{p}$ and $d_{i,j,q}$ $\in \bC$ satisfy
\begin{align}
    &d_{i,j,p} = d_{j,i,p}, \label{15} \\
    &\sum_{q \in \mathbb{Z}} f(M_q) d_{i,j,q} = \alpha f(M_i) f(M_j), \label{19} \\
    &\sum_{q \in \mathbb{Z}} d_{r,s,q} d_{q,t,p} = \sum_{q \in \mathbb{Z}} d_{s,t,q} d_{q,r,p}  \label{18}
\end{align}
for any $i, j, p, r, s, t \in \mathbb{Z}$.

Moreover, when $\alpha=0$ and $c_p=0$ for any $p\in\bZ$, $(A_{f,k}, \cdot, [\cdot,\cdot,\cdot])$ is a Poisson $3$-Lie algebra; otherwise, it is not.
\end{theorem}

\begin{proof}
For brevity, we omit the multiplication symbol $\cdot$ in the proof. We aim to describe the multiplication $\cdot$. Let $(A_{f,k},\cdot,[\cdot,\cdot,\cdot])$ be an transposed Poisson structure defined on 3-Lie algebra $A_{f,k}$. By Lemma \ref{jhjh}, for every element $x \in A_{f,k}$, there is a related $\frac{1}{3}$-derivation $\varphi_x$ of $A_{f,k}$ such that $\varphi_x(y)=x \cdot y$. For any $i\in \bZ$, we set
\[\varphi_{L_i}:
\begin{cases}
\varphi_{L_i}(L_r)=h_{L_i}L_r,\\
\varphi_{L_i}(M_r)=f(M_r)\sum\limits_{p\in\bZ}{a_{i,p}L_p}+\sum\limits_{q\in\bZ}{b_{i,r,q}M_q},
\end{cases}
\]
\[\varphi_{M_i}:
\begin{cases}
\varphi_{M_i}(L_r)=h_{M_i}L_r,\\
\varphi_{M_i}(M_r)=f(M_r)\sum\limits_{p\in\bZ}{c_{i,p}L_p}+\sum\limits_{q\in\bZ}{d_{i,r,q}M_q},
\end{cases}
\]
where \( h_{L_i} \), \( h_{M_i} \), \( b_{i,r,q} \) and \( d_{i,r,q} \) satisfy

$$\sum_{q \in \mathbb{Z}} f(M_q) b_{i,r,q} = h_{L_i} f(M_r),  $$
\begin{equation}\label{eq:9}
\sum_{q \in \mathbb{Z}} f(M_q) d_{i,r,q} = h_{M_i} f(M_r).
\end{equation}

Since
$$L_iL_j=\varphi_{L_i}(L_j)=h_{L_i}L_j,\quad L_jL_i=\varphi_{L_j}(L_i)=h_{L_j}L_i,\quad \forall i,j \in \bZ,$$
we have
\begin{equation}
h_{L_i} = h_{L_j}=0,
\end{equation}
hence,
$$L_iL_j=0.$$

Since
$$L_iM_j=\varphi_{L_i}(M_j)=f(M_j)\sum\limits_{p\in\bZ}{a_{i,p}L_p}+\sum\limits_{q\in\bZ}{b_{i,j,q}M_q},\quad M_jL_i=\varphi_{M_j}(L_i)=h_{M_j}L_i,$$
it follows that, for any $i,j,q\in\bZ$,
\begin{align}
&b_{i,j,q}=0, \label{eq:11} \\
&h_{M_j}=f(M_j)a_{i,i}, \label{eq:12}
\end{align}
and when $i\neq j$
\begin{equation}
a_{i,j}=0.
\end{equation}
By Eq.\eqref{eq:12}, we have $a_{i,i}$ is a number that is independent of $i$, and we denote it by $\alpha$,
hence,
$$L_iM_j=\alpha f(M_j)L_i.$$

From Eq.\eqref{eq:9} and \eqref{eq:12}, we have
$$\sum_{q \in \mathbb{Z}} f(M_q) d_{i,j,q} = \alpha f(M_i) f(M_j),$$
i.e., we get the Eq.\eqref{19}.

Since
$$M_iM_j=\varphi_{M_i}(M_j)=f(M_j)\sum_{p\in\bZ}{c_{i,p}L_p}+\sum_{q\in\bZ}{d_{i,j,q}M_q}, \quad M_jM_i=\varphi_{M_j}(M_i)=f(M_i)\sum_{p\in\bZ}{c_{j,p}L_p}+\sum_{q\in\bZ}{d_{j,i,q}M_q},$$
we have
\begin{equation*}
d_{i,j,q}=d_{j,i,q}
\end{equation*}
and
\begin{equation}\label{ccccc}
f(M_j)c_{i,p}=f(M_i)c_{j,p},
\end{equation}
then there exists a number $ c_p $ such that $ c_{i,p} = c_pf(M_i) $ for any $i\in\bZ$, hence,
$$M_i M_j=f(M_i)f(M_j)\sum\limits_{p\in\bZ}{c_pL_p}+\sum\limits_{q\in\bZ}{d_{i,j,q}M_q}.$$

In summary, if $(A_{f,k},\cdot,[\cdot,\cdot,\cdot])$ is a transposed Poisson $3$-Lie algebra, the multiplication table is given by
\begin{align*}
&L_i L_j = 0,  \\
&L_i M_j=\alpha f(M_j)L_i ,  \\
&M_i M_j=f(M_i)f(M_j)\sum_{p\in\bZ}{c_pL_p}+\sum_{q\in\bZ}{d_{i,j,q}M_q}
\end{align*}
for any $i,j \in \bZ$ and the coefficients satisfy Eq.\eqref{15} and \eqref{19}.

Considering the associative identity \((X Y)  Z = X (Y Z)\), we have the following discussion.

For $L_r,M_s,M_t \in A_{f,k}$, we have
\begin{align*}
(L_r  M_s)  M_t &= \alpha f(M_s) L_r  M_t = \alpha ^ 2f(M_s) f(M_t) L_r, \\
L_r  (M_s  M_t) &= \sum_{q \in \mathbb{Z}} d_{s,t,q} L_r  M_q = \left( \sum_{q \in \mathbb{Z}} f(M_q) d_{s,t,q} \right) \alpha L_r = \alpha ^ 2f(M_s) f(M_t) L_r,
\end{align*}
thus $(L_r  M_s)  M_t=L_r  (M_s  M_t)$ holds.\\

For $M_r,M_s,M_t \in A_{f,k}$, we have
\begin{align*}
&(M_r  M_s)  M_t = \left( f(M_s)f(M_r)\sum_{p\in\bZ}c_{p}L_p + \sum_{q\in\bZ}d_{r,s,q}M_q \right)  M_t \\
=& f(M_t)f(M_s)f(M_r)\alpha \sum_{p\in\bZ}c_{p}L_p + f(M_t)\sum_{q\in\bZ}\sum_{p\in\bZ}f(M_q)d_{r,s,q}c_{p}L_p + \sum_{q\in\bZ}\sum_{p\in\bZ} d_{r,s,q}d_{q,t,p}M_p\\
=& f(M_t)f(M_s)f(M_r)\alpha \sum_{p\in\bZ}c_{p}L_p + f(M_t)\sum_{p\in\bZ}(\sum_{q\in\bZ}f(M_q)d_{r,s,q})c_{p}L_p + \sum_{q\in\bZ}\sum_{p\in\bZ} d_{r,s,q}d_{q,t,p}M_p\\
\overset{\text{\eqref{19}}}{=}& f(M_t)f(M_s)f(M_r)\alpha\sum_{p\in\bZ} c_{p}L_p + f(M_t)f(M_s)f(M_r)\alpha\sum_{p\in\bZ} c_{p}L_p + \sum_{q\in\bZ}\sum_{p\in\bZ} d_{r,s,q}d_{q,t,p}M_p.
\end{align*}
\begin{align*}
&M_r  (M_s  M_t) = M_r  \left( f(M_t)f(M_s)\sum_{p \in \bZ} c_{p} L_p + \sum_{q \in \bZ} d_{s,t,q} M_q \right) \\
=& f(M_r) f(M_t) f(M_s)\alpha\sum_{p \in \bZ} c_{p} L_p + f(M_r)\sum_{q \in \bZ} \sum_{p \in \bZ} f(M_q) d_{s,t,q} c_{p} L_p + \sum_{q \in \bZ} \sum_{p \in \bZ} d_{s,t,q} d_{r,q,p} M_p\\
=& f(M_r) f(M_t) f(M_s)\alpha\sum_{p \in \bZ} c_{p} L_p + f(M_r)\sum_{p \in \bZ} (\sum_{q \in \bZ} f(M_q) d_{s,t,q}) c_{p} L_p + \sum_{q \in \bZ} \sum_{p \in \bZ} d_{s,t,q} d_{r,q,p} M_p\\
\overset{\text{\eqref{19}}}{=}& f(M_r) f(M_t) f(M_s)\alpha\sum_{p \in \bZ} c_{p} L_p + f(M_r)f(M_s)f(M_t)\alpha\sum_{p \in \bZ}c_{p} L_p + \sum_{q \in \bZ} \sum_{p \in \bZ} d_{s,t,q} d_{r,q,p} M_p.
\end{align*}

Since $(M_r  M_s)  M_t=M_r  (M_s  M_t)$, comparing the coefficients of $M_p$, we have
\begin{equation*}
\sum_{q\in\bZ}{d_{r,s,q}d_{q,t,p}}=\sum_{q\in\bZ}{d_{s,t,q}d_{r,q,p}},
\end{equation*}
that is, Eq.\eqref{18} holds.

For any $L_r,M_s,M_t \in A_{f,k}$, we have
\begin{equation}\label{lmm1}
(L_rM_s) M_t=\alpha f(M_s)L_r M_t=\alpha^2 f(M_s)f(M_t)L_r
\end{equation}
and
\begin{equation}\label{lmm2}
L_r (M_s M_t)=L_r \left(f(M_s)f(M_t) \sum_{p \in \mathbb{Z}} c_{p} L_p + \sum_{q \in \mathbb{Z}} d_{s,t,q} M_q\right)\sum_{q \in \mathbb{Z}}\alpha f(M_q)d_{s,t,q}L_r\overset{\text{\eqref{19}}}{=}\alpha^2 f(M_s)f(M_t)L_r,
\end{equation}
so $(L_r M_s) M_t=L_r(M_s M_t)$ holds.

For any $M_r,L_s,M_t \in A_{f,k}$, by Eq.\eqref{lmm1}, we have
$$(M_r  L_s) M_t=(L_s M_r) M_t=\alpha^2 f(M_r)f(M_t)L_s$$
and
$$M_r (L_s M_t)=(L_s M_t) M_r=\alpha^2 f(M_t)f(M_r)L_s,$$
so $(M_r L_s) M_t=M_r (L_s M_t)$ holds.

Considering $M_r,M_s,L_t \in A_{f,k}$, by Eq.\eqref{lmm2}, we have
$$(M_r M_s)L_t=L_t (M_rM_s)=\alpha^2 f(M_r)f(M_s)L_t.$$
By Eq.\eqref{lmm1},
$$M_r( M_s L_t)=(L_tM_s)M_r=\alpha^2 f(M_s)f(M_r)L_t,$$
so $(M_r M_s) L_t=M_r (M_s L_t)$ holds.

The remaining cases are as follows
\begin{align*}
(L_r L_s)L_t=L_r( L_s L_t),\\
(L_r L_s) M_t=L_r( L_s M_t),\\
(L_r M_s) L_t=L_r( M_s L_t),\\
(M_r L_s) L_t=M_r ( L_s L_t).
\end{align*}
Clearly, they hold naturally.
In summary, we obtain that the transposed Poisson structures defined on $A_{f,k}$ are given by
\begin{align*}
&L_i L_j = 0,  \\
&L_i M_j=\alpha f(M_j)L_i ,  \\
&M_i M_j=f(M_i)f(M_j)\sum_{p\in\bZ}{c_{p}L_p}+\sum_{q\in\bZ}{d_{i,j,q}M_q},
\end{align*}
where $\alpha$ and $d_{i,j,q}$ satisfy Eq.\eqref{15}, \eqref{19} and \eqref{18}.

Next, we consider whether the transposed Poisson 3-Lie algebra $(A_{f,k},\cdot,[\cdot,\cdot,\cdot])$ is a Poisson $3$-Lie algebra.

If $(A_{f,k},\cdot,[\cdot,\cdot,\cdot])$ is a Poisson $3$-Lie algebra.

By Eq.\eqref{poisson3}, for any $M_r,M_s,L_t,M_q \in A_{f,k}$, we have
\begin{equation*}
[M_rM_s,L_t,M_q]=f(M_r)f(M_s)f(M_q)\sum_{p\in\bZ}{c_p(p-t)L_{p+t+k}}=[M_r,L_t,M_q]M_s+M_r[M_s,L_t,M_q]=0,
\end{equation*}
comparing coefficients of the basis elements we obtain
\begin{equation}
f(M_r)f(M_s)f(M_q)c_p(p-t)=0.
\end{equation}
There exist $r,s,q,t\in\bZ$ such that $f(M_r)$, $f(M_s)$, $f(M_q)$ and $p-t$ are all nonzero, hence $c_p=0$ for any $p\in\bZ$.

By Eq.\eqref{poisson3}, for any $M_r,M_s,L_t,L_q\in A_{f,k}$, we have
\begin{equation*}
\begin{split}
&[M_rM_s,L_t,L_q]
=f(M_r)f(M_s)\sum_{p\in\bZ}{c_p[L_p,L_t,L_q]}+\sum_{p\in\bZ}{d_{r,s,p}[M_p,L_t,L_q]}\\
=&(t-q)\sum_{p\in\bZ}{d_{r,s,p}f(M_p)L_{t+q+k}}
\overset{\text{\eqref{19}}}{=}(t-q)\alpha f(M_r)f(M_s)L_{t+q+k},
\end{split}
\end{equation*}
on the other hand,
\begin{equation*}
\begin{split}
&[M_rM_s,L_t,L_q]\\
=&[M_r,L_t,L_q]M_s+M_r[M_s,L_t,L_q]
=f(M_r)(t-q)L_{t+q+k}M_s+f(M_s)(t-q)L_{t+q+k}M_r\\
=&(t-q)\alpha f(M_r)f(M_s)L_{t+q+k}+(t-q)\alpha f(M_s)f(M_r)L_{t+q+k}=2(t-q)\alpha f(M_r)f(M_s)L_{t+q+k},
\end{split}
\end{equation*}
comparing coefficients of the basis elements we obtain
\begin{equation}
(t-q)\alpha f(M_r)f(M_s)=0.
\end{equation}
There exist $r,s,t,q\in\bZ$ such that $f(M_r)$, $f(M_s)$ and $t-q$ are all nonzero, then $\alpha=0$.

By Eq.\eqref{poisson3}, for $L_r,M_s,L_t,M_q\in A_{f,k}$, we have
\begin{equation*}
\begin{split}
&[L_rM_s,L_t,M_q]=\alpha f(M_s)[L_r,L_t,M_q]=(r-t)\alpha f(M_q)f(M_s)L_{r+t+k}\\
=&[L_r,L_t,M_q]M_s+L_r[M_s,L_t,M_q]=f(M_q)(r-t)L_{r+t+k}M_s=(r-t)\alpha f(M_q)f(M_s)L_{r+t+k}.
\end{split}
\end{equation*}
The remaining cases are as follows
\begin{align*}
&[L_rL_s,L_t,L_q]=[L_r,L_t,L_q]L_s+L_r[L_s,L_t,L_q],\\
&[L_rL_s,L_t,M_q]=[L_r,L_t,M_q]L_s+L_r[L_s,L_t,M_q],\\
&[L_rL_s,M_t,M_q]=[L_r,M_t,M_q]L_s+L_r[L_s,M_t,M_q],\\
&[L_rM_s,M_t,M_q]=[L_r,M_t,M_q]M_s+L_r[M_s,M_t,M_q],\\
&[L_rM_s,L_t,L_q]=[L_r,L_t,L_q]M_s+L_r[M_s,L_t,L_q],\\
&[M_rM_s,M_t,M_q]=[M_r,M_t,M_q]M_s+M_r[M_s,M_t,M_q].
\end{align*}
It is easy to see that all the above equations hold. Thus, if $\alpha=0$ and $c_p=0$ for any $p\in\bZ$, $(A_{f,k}, \cdot, [\cdot,\cdot,\cdot])$ is a Poisson $3$-Lie algebra. It gives the complete statement of the theorem.
\end{proof}
Next, we present a concrete example about the transposed Poisson structures defined on the 3-Lie algebra $A_{f,k}$.

\begin{example}
Let $\{d_i,i\in \bZ\}$ be a nonzero sequence of complex numbers such that
$d_i = 0$ for all but a finite number of $i$. For any $i,j$ and $p$ $\in \bZ$, we set
$$d_{i,j,p}=d_pf(M_i)f(M_j),$$
then $d_{i,j,p}$ satisfy Eq.\eqref{15}, \eqref{19} and \eqref{18}, we have
$$\sum_{q \in \mathbb{Z}} f(M_q) d_q = \alpha,$$
then we get the following transposed Poisson structure on $A_{f,k}$
\begin{align*}
&L_i L_j=0,\\
&L_i M_j=\alpha f(M_j)L_i,\\
&M_i M_j=f(M_i)f(M_j)( \sum_{p \in \mathbb{Z}} c_{p} L_p + \sum_{q \in \mathbb{Z}} d_q M_q).
\end{align*}

\end{example}
\begin{center}
{\bf Funding}
\end{center}

This work is supported by the NSFC(12201624) and the Scientific Research Foundation of Civil Aviation University of China (3122023PT20).

\begin{center}
{\bf  Data availability}
\end{center}

No data was used for the research described in the article.

\renewcommand{\refname}{References}

\end{document}